\documentclass[preprint]{elsarticle}

\usepackage{lscape}
\usepackage{url}

\usepackage{a4wide}
\usepackage{enumerate}

\usepackage{a4wide}
\usepackage[parfill]{parskip}

\usepackage{graphicx}

\usepackage{amssymb}

\usepackage{amsthm}

\usepackage{amsmath}
\usepackage{enumerate}

\usepackage{multirow}
\usepackage{setspace}
\usepackage{float}

\newtheorem{theorem}{Theorem}[section]

\newtheorem{lemma}[theorem]{Lemma}
\newtheorem{claim}[theorem]{Claim}
\newtheorem{corollary}[theorem]{Corollary}
\newtheorem{conjecture}[theorem]{Conjecture}

\newtheorem{problem}[theorem]{Problem}

\newcommand{\tpp}[1]{\ensuremath{\mathcal{T}_{#1}}}
\newcommand{\pp}[1]{\ensuremath{\mathcal{P}_{#1}}}

\usepackage[margin=20pt, small, bf]{caption}

\newcommand{\hyper}{\mathcal}

\journal{a journal}

\title{A family of extremal hypergraphs for Ryser's conjecture}

\author[imp]{A.~Abu-Khazneh}
\ead{a.abu-khazneh@imperial.ac.com}
\address[imp]{Faculty of Natural Sciences, Imperial College London, Exhibition Road, SW7 2AZ London, UK}

\author[pan]{J\'{a}nos Bar\'{a}t\fnref{JSup}}
\ead{barat@cs.elte.hu}
\address[pan]{Department of Mathematics, University of Pannonia, 8200 Veszpr\'em, Hungary\\ Geometric and Algebraic Combinatorics Research Group, MTA-ELTE, H--1117 Budapest, Hungary}
\fntext[JSup]{Supported by Sz\'echenyi 2020 under the EFOP-3.6.1-16-2016-00015, and OTKA-ARRS Slovenian-Hungarian Joint Research Project, grant no. NN-114614.}

\author[eth]{A.~Pokrovskiy\corref{cor}\fnref{ASup}}
\ead{dr.alexey.pokrovskiy@gmail.com}
\cortext[cor]{Corresponding author.}
\fntext[ASup]{Research supported in part by SNSF grant 200021-175573.}
\address[eth]{Department of Mathematics, ETH Z\"urich, 8092 Z\"urich, Switzerland}

\author[fur]{Tibor Szab\'{o}}
\ead{szabo@math.fu-berlin.de}
\address[fur]{Institut f\"ur Mathematik, Freie Universit\"at, 14195 Berlin, Germany}

\begin{document}

\begin{abstract}
Ryser's Conjecture states that for any $r$-partite $r$-uniform 
hypergraph, the vertex cover number is at most $r{-}1$ times the matching number. 
This conjecture is only known to be true for $r\leq 3$ in general and for $r\leq 5$ 
if the hypergraph is intersecting. 
There has also been considerable effort made for finding 
hypergraphs that are extremal for Ryser's Conjecture, 
i.e. $r$-partite hypergraphs whose cover number is $r-1$ times
its matching number.
Aside from a few sporadic examples, the set of uniformities $r$ for
which Ryser's Conjecture is known to be tight is 
limited to those integers for which a projective plane of order $r-1$
exists. 

We produce a new infinite family of $r$-uniform hypergraphs extremal to Ryser's
Conjecture, which exists whenever a projective plane of order $r-2$ exists. 
Our construction is flexible enough to produce a large number of non-isomorphic 
extremal hypergraphs.  
In particular, we define what we call the {\em Ryser poset} of 
extremal intersecting $r$-partite $r$-uniform hypergraphs and show
that the number of maximal and minimal elements is exponential in
$\sqrt{r}$.

This provides further evidence for the difficulty of Ryser's Conjecture.  
\end{abstract}


\maketitle
\section{Introduction}

A {\em cover} of a hypergraph is a set of vertices meeting every edge of the hypergraph. 
The vertex {\it cover number} $\tau({\cal H})$ of a hypergraph ${\cal H}$ is the number of vertices in the smallest cover of ${\cal H}$. 
A {\em matching} is a set of disjoint edges, and the {\it matching number} $\nu ({\cal H})$ 
of a hypergraph ${\cal H}$ is the  maximum size of a matching consisting of edges of 
${\cal H}$.   A hypergraph with $\nu({\cal H})=1$ is called {\it intersecting}.

A hypergraph is $r$-uniform if every edge has $r$ vertices. 
Any $r$-uniform hypergraph ${\cal H}$ satisfies the inequality $\tau ({\cal H})\le r 
\nu({\cal H})$, since  the union of  the edges of a maximum matching is a cover.
This bound is sharp, as shown by the family of all subsets of size $r$ in a ground set of size $kr-1$ which has 
$\nu=k-1$ and $\tau=(k-1)r$.
There is another sharp example for $\nu =1 $:  any $r$-uniform hypergraph 
consisting of the lines of some projective plane of order $r{-}1$ 
(denoted by $\mathcal{P}_{r}$).
To obtain an example for arbitrary $\nu$, one can take 
the union of disjoint copies of $\mathcal{P}_{r}$. 
A hypergraph is {\it $r$-partite} if its vertex set $V$ can be partitioned into $r$ sets $V_1,\dots, V_r$, called the {\it sides} 
of the hypergraph, so that every edge contains at most one vertex from each side.  
A conjecture commonly attributed to Ryser (but which first appeared in a thesis by his student Henderson ~\cite{HEND,BW}), asserts that the upper bound $\tau ({\cal H})\le r 
\nu({\cal H})$  can be improved if the hypergraph is $r$-partite:

\medskip

\begin{conjecture}\label{RyserConjecture}
  For any $r$-partite $r$-uniform hypergraph we have 
  \begin{equation}\label{RyserBound}
    \tau(\hyper H)\leq (r{-}1) \nu(\hyper H).
  \end{equation}
\end{conjecture}

When $r=2$, Ryser's Conjecture is equivalent to K\"onig's Theorem. 
The only other known general case of the conjecture is 
$r=3$, which was proved by Aharoni \cite{AHR}. 
However, the conjecture is also known to be true for some special cases. 
In particular, it has been proven by Tuza \cite{TUZ} for $r$-partite intersecting hypergraphs when $r \leq 5$, 
and  by Franceti\'c, Herke, McKay, and Wanless \cite{FHMW} for $r \leq 9$,
when one makes the further assumption that any two edges of the 
$r$-partite hypergraph intersect in exactly one vertex.

Besides trying to prove the conjecture, there has also been considerable effort in understanding which hypergraphs are extremal 
for Ryser's Conjecture, i.e. finding $r$-partite hypergraphs $\hyper H$ with $\tau(\hyper H)= (r{-}1) \nu(\hyper H)$. 
We call such an object an {\em $r$-Ryser hypergraph} (or, without
specifying its uniformity, a {\em Ryser hypergraph}).
Denoted by $\mathcal{T}_{r}$, the {\em truncated projective plane} of uniformity $r$ is obtained from $\mathcal{P}_{r}$ by the removal 
of a single vertex $v$ and the lines containing $v$. 
The sides $V_1,\dots, V_r$ of $\mathcal{T}_{r}$ are the sets of vertices other than 
$v$ on the lines containing $v$. 
It is known and not difficult to see that $\mathcal{T}_r$ is 
intersecting and its cover number is one less than its uniformity $r$.
Except for finitely many sporadic examples, all minimal hypergraphs known to attain Ryser's bound 
are subhypergraphs of truncated projective planes. Consequently, aside from finitely many exceptions, the set of uniformities 
$r$ for which Ryser's Conjecture is known to be tight is limited to those  integers for which
a projective plane of order $r{-}1$ exists.

Finite projective planes are only known to exist for orders that are prime powers, and it is 
a long-standing open problem to decide 
whether there exists a projective plane of any other order.
A few non-existence results are known about projective planes, in particular 
it has been shown that finite projective 
planes of order $6$ and $10$ do not exist \cite{BRRY, LTS}. 
This implies that the first values of $r$, for which the truncated projective plane construction of uniformity $r$ does not work are $7$ and $11$.
Inspired by the lack of examples attaining Ryser's bound for these values, 
Aharoni, Bar\'{a}t and Wanless~\cite{ABW} constructed $7$-partite intersecting hypergraphs with cover number $6$. 
This was also obtained independently by Abu-Khazneh and Pokrovskiy~\cite{AKP}, who also constructed an 
$11$-partite intersecting hypergraph with cover number $10$. In~\cite{FHMW} Franceti\'c, Herke, McKay, and Wanless  constructed a $13$-partite intersecting hypergraph with cover number $12$.

\subsection{Results} 

Our main goal is to construct intersecting $r$-Ryser hypergraphs 
for an infinite sequence of uniformities $r$ such that
$r{-}1$ is {\em not} a prime power. 
We prove the following theorem.

\medskip

\begin{theorem}~\label{MainTheorem}
Let $\mathcal T$ be an $r$-partite $r$-uniform intersecting hypergraph, and let 
$S\in \mathcal T$ be an edge such that 
$S$ intersects every other edge in one vertex, $\tau(\mathcal T-S) =r{-}1$, 
and the only covers of $\mathcal{T}-S$ of size $r{-}1$ are sides.
Then there exists an intersecting $(r+1)$-Ryser hypergraph $\hyper H$.
\end{theorem}

By the aid of the following Lemma, it is easy to see that for $r\geq 4$, the truncated projective plane
$\mathcal{T}_r$ together with an arbitrary hyperedge $S \in\mathcal{T}_r$ 
satisfy the conditions in Theorem~\ref{MainTheorem}.\\

\begin{lemma}~\label{TruncatedCovering}
  If $r \geq 4$, and $W \subset V(\mathcal{T}_r)$ such that $|W| = r
  -1$ and $W$ contains vertices from at least two sides of $\mathcal{T}_r$, 
  then $W$ covers at most $|\mathcal{T}_r| - 2$ of the edges of $\mathcal{T}_r$.
\end{lemma}

\begin{proof}
  It follows from the axioms of finite projective planes that
  $\mathcal{T}_r$ is an $(r-1)$-regular hypergraph with $(r-1)^2$
  edges.   Since $|W|\geq r-1 \geq 3$, by the conditions of the lemma $W$ must
  contain three vertices $w,v', v'' \in W$ such that $w$ is in not in
  the same side as the other two vertices $v'$ and $v''$. 
  It then also follows from the axioms of finite projective planes
  that $w$ is contained in an edge of $\mathcal{T}_r$ that contains
    $v'$, and in an edge that contains $v''$ (these two edges might be
    the same). In any case $w$ covers a common edge with $v'$ and a
    common edge with $v''$.  
Hence, $W$ will cover at most $(r-1)^2 - 2$ edges of $\mathcal{T}_r$. 
\end{proof}

Since explicit constructions of $\mathcal{T}_r$ are known when $r{-}1=q$ is  a prime
power, the following is immediate.

\medskip

\begin{corollary}\label{cor:MainCorollary}
  For any prime power $q$, there exists an intersecting  $(q+2)$-Ryser hypergraph.
\end{corollary}

 Using this, one can argue that 
there are infinitely many values of uniformities for which Corollary~\ref{cor:MainCorollary} 
gives  a hypergraph that is tight for the bound in Ryser's Conjecture,   
but for which there can be no truncated projective plane
construction (see Section~\ref{sec:conclusion}).

The truncated projective plane construction provides a ready supply of graphs that satisfy the conditions of Theorem \ref{MainTheorem}. 
Kahn~\cite{KAH} proved that with high probability a 
randomly chosen $22 r \log r$ lines of $\pp{r}$ cannot be covered with less than $r$ points.
Boros, Sz\H onyi and Tichler~\cite{BSzT} modified Kahn's idea to show that the only covers of size $r$ are lines in this case. 
Therefore, Theorem~\ref{MainTheorem} can be applied to these settings as well.
However, the applications of the theorem is not limited to subgraphs of projective planes. 
For instance, it can be used on the following $8$-partite hypergraph, which was constructed in \cite{FHMW}, where it was labelled $\mathcal{H}_{38}$, and where it was proved that $\mathcal{H}_{38}$ is an $8$-Ryser hypergraph that is also intersecting and linear, but is not a subhypergraph of $\mathcal{T}_8$.

\begin{center}
  \doublespacing
  {
    
    \small
    \singlespace
    \begin{tabular}{l l l l l l l l }

      \texttt{bcdefgaa} & \texttt{begcdafb} & \texttt{bfegadcc} & \texttt{bgafcedd} & \texttt{cgbafdeb} & \texttt{cadgbefe} & \texttt{cdefgaba} & \texttt{cefbagdf} \\
      \texttt{dfbcgeaf} & \texttt{dgfebacg} & \texttt{dcgfabee} & \texttt{deagfcbh} & \texttt{efgabcda} & \texttt{edfgcbab} & \texttt{eabfdgch} & \texttt{ecabgdfg} \\
      \texttt{fedagbcd} & \texttt{fgebdcae} & \texttt{fagecdbf} & \texttt{fdacbgec} & \texttt{gcfadebc} & \texttt{gaecfbdg} & \texttt{gfdbcaeh} & \texttt{gdbeacfd} \\
      \texttt{aaaaaaai} & \texttt{bbbbbbbi} & \texttt{ccccccci} & \texttt{dddddddi} & \texttt{eeeeeeei} & \texttt{fffffffi} & \texttt{gggggggi} & \texttt{adgbfecj} \\
                        & \texttt{dbeacgfj} & \texttt{gecfbdaj} & \texttt{bafdgcej} & \texttt{fcbgeadj} & \texttt{egdcafbj} & \texttt{cfaedbgj} &                   \\ \\ 

    \end{tabular}
  }

\end{center}

We verified using the help of a computer that all 24 edges in the first three rows of the above presentation of $\mathcal{H}_{38}$ satisfy the conditions of edge $S$ in the statement of Theorem \ref{MainTheorem}, and consequently each allowing $\mathcal{H}_{38}$ to be used to construct a $9$-partite extremal hypergraph. 

Moreover, apart from the above direct applications, the construction that we use to prove Theorem~\ref{MainTheorem} is flexible enough 
to give not only one, but many non-isomorphic Ryser hypergraphs. 
This highlights the difficulty of Ryser's Conjecture, since a proof of it would eventually need to deal with all these extremal 
constructions. 
Define an intersecting $r$-Ryser hypergraph to be \emph{minimal} if the deletion of any edge produces a hypergraph with cover number $r-2$. 
We prove that there are many minimal intersecting $r$-Ryser hypergraphs.

\medskip 

\begin{theorem}\label{counting}
There is an infinite sequence of integers $r$, for which there are 
$\exp (r^{0.5 - o(1)})$ non-isomorphic minimal intersecting $r$-Ryser hypergraphs.
\end{theorem}

The notion of containment-maximal Ryser hypergraphs 
turns out to be more subtle, since these hypergraphs may be infinite. 
It is nevertheless possible to give a meaningful definition of the
concept, and prove that the number of maximal intersecting $r$-Ryser
hypergraphs is exponential in $\sqrt{r}$. We postpone
 the precise statement of the relevant theorem to~Section~\ref{sec:maximal}.

\section{New extremals from old} \label{sec:newextremals}
In this section, we prove Theorem~\ref{MainTheorem}. 
We define an $\{ r{-}1, r\}$-uniform hypergraph to be a family of sets of size $r{-}1$ and 
$r$. 
Notice that in order to find an $r$-uniform hypergraph $\mathcal H$ with 
$\tau(\mathcal H)=r{-}1$, 
it suffices to find an $\{ r{-}1, r\}$-uniform $\hyper H'$ with $\tau(\mathcal H')=r{-}1$.
Once we have such a hypergraph, we can construct an $r$-uniform hypergraph from 
$\hyper H'$ by adding a separate new vertex to each edge of size
$r{-}1$. For the rest of the paper, we also sometimes abuse notation
to enhance readability by omitting braces around singleton vertex
sets, so we write $F - s$ instead of $F\setminus \{s\}$ and $F+s$
instead of $F\cup \{ s\}$. 

Let $\mathcal T$ be an $r$-partite $r$-uniform intersecting hypergraph with sides $V_1, \dots, V_r$. 
Let $S = \{s_1,\ldots, s_r\}$ be an edge of $\mathcal{T}$, with $s_i\in V_i$, 
that satisfies the conditions in Theorem~\ref{MainTheorem}.
Let $F_1, \dots, F_r$  be  $r$ edges of $\mathcal T$ with $s_i\in F_i\cap S$ for each $i$, and also  $(F_i- s_i)\cap (F_j- s_j)\neq \emptyset$ for all $i,j$. 
The edges $F_1, \dots, F_s$ do not have to be distinct --- one possibility is to take $F_1=\dots=F_r=S$.

We define an $\{ r, r{+}1\}$-uniform, intersecting hypergraph 
$\mathcal H(\mathcal T, S, F_1, \dots, F_r)$,  which has cover number $r$.

\begin{itemize}
\item The vertex set of $\mathcal H(\mathcal T, S, F_1, \dots, F_r)$ consists  of 
the vertex set of $\mathcal T$ together with $r$ vertices $v_1, \dots, v_r$ in side $V_{r+1}$.

\item For an edge $E\neq S$ of $\mathcal T$ satisfying $E\cap S= \{s_i\}$, we define 
$\hat E= E+v_i$.
That is, $\hat E$ is an $(r+1)$-edge built from $E$ by adding the vertex $v_i$ corresponding to the vertex of $S$ which $E$ contains. 
Notice that $\hat E$ is well-defined since $S$ intersects any other edge of $\mathcal{T}$ 
in exactly one vertex.

\smallskip
  
Define 
\begin{align*}
\mathcal E_1 &=\{\hat E: E\in \hyper T-S\}\\
\mathcal E_2 &=\{F_i: i=1,\dots, r\}\\
\mathcal E_3 &= \{F_i-s_i+v_i:i=1,\dots,r\}.
\end{align*}
We let $\mathcal H(\mathcal T, S, F_1, \dots, F_r) = 
\mathcal E_1\cup\mathcal E_2\cup\mathcal E_3$.
\end{itemize}

In other words $\mathcal H(\mathcal T, S, F_1, \dots, F_r)$ has three parts: 
The first part consists of taking the $(r+1)$-edges $\hat E$ for all $E\in \mathcal T$ 
other than $S$. 
The second part consists of the $r$-edges $F_i$. 
The third part consists of the $r$-edges created from $F_1, \dots, F_r$ 
by deleting for each $F_i$ the designated vertex $s_i$ from its
intersection with $S$, 
and then adding to it the corresponding vertex $v_i$.

First we show that these hypergraphs are intersecting.

\medskip 

\begin{lemma}\label{LemmaIntersecting}
$\mathcal H(\mathcal T, S, F_1, \dots, F_r)$ is intersecting.
\end{lemma}
\begin{proof}
The hypergraph induced by $\mathcal E_1 \cup \mathcal E_2$ is intersecting since its restriction to the first $r$ sides gives 
a subhypergraph of $\mathcal T$, which is an intersecting hypergraph. 
Furthermore for any $i$ and $j$, we have 
$(F_i-s_i+v_i)\cap (F_j-s_j+v_j) \supseteq (F_i- s_i)\cap (F_j- s_j) \neq \emptyset$ by assumption.
Therefore the hypergraph induced by $\mathcal E_2 \cup \mathcal E_3$ is intersecting.

It remains to show that edges in $\mathcal E_1$ intersect those in $\mathcal E_3$.
That is, $\hat E\cap (F_i-s_i+v_i)\neq \emptyset$ for any 
$E\neq S$ and $i=1, \dots, r$. 
Since $\mathcal T$ is intersecting, there is some vertex $x\in E\cap F_i$. 
If $x\neq s_i$, then $x\in \hat E\cap (F_i-s_i+v_i)$. 
Otherwise $v_i\in  \hat E\cap (F_i-s_i+v_i)$.
\end{proof}

We show that the covers of the hypergraph  $\mathcal H(\mathcal T, S, F_1, \dots, F_r)$  
have a very specific structure.

\medskip 

\begin{lemma}\label{CoverMirror}
If $C$ is a cover of $\mathcal H(\mathcal T, S, F_1, \dots, F_r)$, 
then $C'= (C\cup \{s_i: v_i\in C\}) \setminus \{v_1, \dots, v_r\}$ 
is a cover of $\hyper T-S$.
\end{lemma} 
\begin{proof}
Let $E$ be an arbitrary edge of $\hyper T-S$. 
We show that $E\cap C'\neq\emptyset.$ 
We know that $C\cap \hat E\neq \emptyset$, since 
$C$ is a cover of $\mathcal H(\mathcal T, S, F_1, \dots, F_r)$. 
Let $y$ be a vertex in $C\cap \hat E$. If $y\not \in\{v_1, \dots, v_r\}$, 
then $y\in C'$ which implies $C'\cap E\neq \emptyset$.
Otherwise $y = v_i$ for some $i$, which implies that $s_i\in C'\cap E$.
\end{proof}

We now prove that $\mathcal H(\mathcal T, S, F_1, \dots, F_r)$ has cover number $r$. 
This immediately implies Theorem~\ref{MainTheorem} (by taking $\mathcal{H}$ to be 
$\mathcal H(\mathcal T, S, F_1, \dots, F_r)$ with a new vertex added to each of its 
$r$-edges).

\medskip 

\begin{theorem}
The hypergraph $\mathcal H(\mathcal T, S, F_1, \dots, F_r)$ is an $(r+1)$-partite,
$\{ r, r+1\}$-uniform, intersecting hypergraph with  
$\tau(H(\mathcal T, S, F_1, \dots, F_r))=r$.
\end{theorem}

\begin{proof}
It is immediate that $\mathcal H(\mathcal T, S, F_1, \dots, F_r)$ is $(r+1)$-partite and
$\{ r, r+1\}$-uniform --- the $r$-edges $E\in {\cal T} -S$ 
just gained a vertex $v_i$ in $V_{r+1}$ in order to become $\hat E$, 
whereas the $r$-edges $F_i$ had vertex $s_i$ deleted in $V_i$, and  
$v_i$ added in $V_{r+1}$. From Lemma~\ref{LemmaIntersecting},  $\mathcal H(\mathcal T, S, F_1, \dots, F_r)$ is intersecting.

It remains to prove that $\mathcal H(\mathcal T, S, F_1, \dots, F_r)$ 
has cover number $r$. 
Suppose to the contrary that there is a cover $C$ of $\mathcal H(\mathcal T, S, F_1, \dots, F_r)$ with $|C|\leq r{-}1$. 
Now $C'$ is a cover of $\mathcal T-S$ by Lemma~\ref{CoverMirror},
and $|C'|\leq |C|\leq r{-}1$. 
By the assumption of Theorem~\ref{MainTheorem}, the cover $C'$ must be one of the 
sides $V_i$ for some $i\in\{1, \dots, r\}$.
Now the definition of $C'$ implies that the cover $C$ is either $V_i$ or $V_i-s_i+v_i$.
In the first case, $C$ does not cover the edge $F_i-s_i+v_i$, while in the second
case, $C$ does not cover the edge $F_i$, both contradicting the 
assumption that $C$ is a cover of 
$\mathcal H(\mathcal T, S, F_1, \dots, F_r)$.
\end{proof}

\section{Many minimal examples} \label{sec:manyextremals}

The goal of  this section is to prove Theorem~\ref{counting}. 
Using the notation of Section \ref{sec:newextremals}, 
let $\mathcal S(\mathcal T_r, S, F_1, \dots, F_r)$ be the following hypergraph:
the vertex set consists of the vertices of $\mathcal T_r$ in sides $V_1,\dots, V_r$ together with $r$ vertices $v_1, \dots, v_r$ in side $V_{r+1}$.
The edge set consists of $\mathcal E_2\cup \mathcal E_3$, where $S=\{s_1,\ldots, s_r\}$ is a fixed hyperedge of $\mathcal T_r$ and 
$F_1, \dots, F_r$  are  $r$ edges of $\mathcal T_r$ with $s_i\in F_i\cap S$ for each $i$, and also  $(F_i- s_i)\cap (F_j- s_j)\neq \emptyset$ for all $i,j$.
First we prove a lemma which
implies that it is sufficient to find many non-isomorphic hypergraphs 
$\mathcal S(\mathcal T_r, S, F_1, \dots, F_r)$.

\medskip

\begin{lemma}\label{Sisomorphism}
Let $S=\{s_1, \dots, s_r\}, F_1, \dots, F_r, G_1, \dots,  G_r$ be
edges of the truncated projective plane $\tpp{r}$ with $S\cap F_i=\{s_i\}=S\cap G_i$.
If there is a subhypergraph ${\cal H}_F$ of 
$\mathcal H(\tpp{r}, S, F_1, \dots, F_r)$ with cover number $\tau({\cal H}_F)=r$ 
which is isomorphic to a subhypergraph ${\cal H}_G$ of 
$\mathcal H(\tpp{r}, S, G_1, \dots, G_r)$ with cover number $\tau({\cal H}_G) = r$,
then  $\mathcal S(\tpp{r}, S, F_1, \dots, F_r)$ is isomorphic to $\mathcal S(\tpp{r}, S, G_1, \dots, G_r)$.
\end{lemma}

\begin{proof}
We claim that $\mathcal S(\tpp{r}, S, F_1, \dots, F_r)$ is contained in $\mathcal H_F$.
Indeed, if an edge $F_i-s_i+v_i$ was missing from ${\cal H}_F$, 
then the side $V_i$ would be a cover of size $r{-}1$.
Similarly, if  $F_i$ was missing, 
then $V_i - s_i+v_i$ would be a cover of size $r{-}1$.
By the same argument, $\mathcal S(\tpp{r}, S, G_1, \dots, G_r)$ is 
contained in $\mathcal H_G$.

Let $\phi$ be an isomorphism from $\mathcal H_F$ to $\mathcal H_G$.
We claim that $\phi$ induces an isomorphism from 
$\mathcal S(\tpp{r}, S, F_1, \dots, F_r)$ to $\mathcal S(\tpp{r}, S, G_1, \dots, G_r)$. 
Indeed, notice that in our construction the possible 
intersection sizes of edges in $\mathcal H(\tpp{r}, S, F_1, \dots, F_r)$  are $1, 2, r{-}1,$ and $r$ 
(since we are assuming $r\geq 4$).
There are only $r$ pairs of hyperedges that have intersection of size $r{-}1$: the 
pairs $F_i$ and $F_i - s_i +v_i$. This implies that any isomorphism must map
a pair of sets $F_i$ and $F_i-s_i+v_i$ into some pair $G_j, G_j-s_j+v_j$ (in some order), 
consequently the restriction of $\phi$ onto the vertex set of 
${\cal S}(\tpp{r},S, F_1, \ldots, F_r)$ is an
isomorphism between this hypergraph and ${\cal S}(\tpp{r},S, G_1, \ldots, G_r)$.
\end{proof}

We will combine Lemma~\ref{Sisomorphism} with the following.

\medskip 

\begin{lemma}\label{CountingS}
For every integer $r$ of the form $q+1$, where $q$ is a prime power, there are at least $\exp(r^{0.5 - o(1)})$
non-isomorphic hypergraphs $\mathcal S(\tpp{r}, S,
F_1, \dots, F_r)$ for the 
different choices of edges  $S, F_1, \dots, F_r\in \tpp{r}$.
\end{lemma}
\begin{proof}
We give a lower bound on the number of non-isomorphic hypergraphs $\mathcal S(\tpp{r}, S, F_1, \dots, F_r)$ by showing that 
there are at least $\exp(r^{0.5 - o(1)})$ distinct degree sequences which can occur in such hypergraphs.

Let us choose $t=\left\lfloor \frac{r^{0.5}}{2} \right\rfloor$ positive integers $x_1, \ldots , x_t$ such that
$t+2 < x_i \leq \lfloor r^{0.5}\rfloor$  for all $i =1, \ldots , t$ and let 
$x_{t+1} = r{-}1 - x_1- \cdots - x_t$. 
We select appropriate edges $F_1, \ldots , F_{r}$ of $\tpp{r}$ for our construction
such that the degree sequence of the non-isolated 
vertices of $V_1$ in $\mathcal S(\tpp{r}, S, F_1, \dots, F_r)$ is $(1, 2x_1, \ldots , 2x_t, 2x_{t+1})$. 
 This is possible since we can partition $S\setminus \{ s_1\}$ 
 into $t+1$ sets $S_1, \ldots , S_{t+1}$ such that $|S_i| = x_i$, and as $F_j$, we 
 select the line connecting  $s_j$ to the $i$th vertex $w_i$ of $V_1\setminus \{ s_1\}$, where
 $j\in S_i$. (Recall that for any  two vertices in different sides of $\tpp{r}$ 
 there is exactly one line passing through both of them.) 
Then the degree of $w_i$ in $\mathcal S(\tpp{r}, S, F_1, \dots, F_r)$ is $2x_i$.
We choose $F_1\neq S$ to be an arbitrary line passing through $s_1$. 
Now the vertex degrees in $V_1$ are as promised. 

We observe that all degrees of $\mathcal S(\tpp{r}, S, F_1, \dots, F_r)$ are at most $2t+4$ in the other sides, since there are
exactly $t+2$ vertices with non-zero degree in the first side (and no pair of vertices is contained in two lines of $\tpp{r}$).
Therefore, the multiset of degrees of the hypergraph $\mathcal S(\tpp{r}, S, F_1, \dots, F_r)$ consists of  
$\{ 2x_1, \ldots , 2x_t, 2x_{t+1}\}$ and some multiset with elements from  $\{ 0, 1, 2, \ldots ,2(t +2)\}$.
Using the fact that each $x_i$ is more than $t+2$, we obtain that each 
multiset $\{ x_1, \ldots, x_t\}$ with elements from $\{ t+3, \ldots, \lfloor r^{0.5} \rfloor\}$ 
corresponds to a different $\mathcal S(\tpp{r}, S, F_1, \dots, F_r)$.

The number of ways to choose the appropriate multiset is at least 
$\binom{t+\lfloor r^{0.5} \rfloor-(t+2)-1}{t} \geq
\Omega \left( \frac{2^{\sqrt{r}}}{\sqrt[4]{r}} \right)$.
\end{proof}

We can now prove Theorem~\ref{counting}.
\begin{proof}[Proof of {Theorem~\ref{counting}}]
For every prime power $q=r-1$
there are at least $\exp\left(r^{0.5-o(1)}\right)$ different 
$(q+2)$-uniform hypergraphs of the form $\mathcal S(\tpp{r}, S, F_1, \dots, F_r)$ by Lemma~\ref{CountingS}. 
Each of these hypergraphs is contained in the corresponding $(r+1)$-Ryser hypergraph 
$\mathcal H(\tpp{r}, S, F_1, \dots, F_r)$. Each 
$\mathcal H(\tpp{r}, S, F_1, \dots, F_r)$ contains some minimal $(r+1)$-Ryser 
hypergraph $\mathcal M(\tpp{r}, S, F_1, \dots, F_r)$. 
These $\exp(r^{0.5-o(1)})$ minimal extremal hypergraphs 
$\mathcal M(\tpp{r}, S, F_1, \dots, F_r)$ are all non-isomorphic by 
Lemma~\ref{Sisomorphism}, since otherwise we would obtain an 
isomorphism between some of the corresponding non-isomorphic 
hypergraphs of the form $\mathcal S(\tpp{r}, S, F_1, \dots, F_r)$.
\end{proof}

\section{Many maximal examples and the Ryser poset}
\label{sec:maximal}
Theorem~\ref{counting} proves that there are many minimal Ryser hypergraphs.
We use the following lemma to show that there are also many non-isomorphic 
maximal intersecting Ryser hypergraphs. 
The lemma states that the hypergraphs $\mathcal H(\tpp{r}, S, F_1,
\dots, F_r)$ are essentially maximal. We mean this in the sense that
there are only ``trivial'' ways to add edges to them:
any new edge must be a ``twin copy" of some $\hat{F_i}$, differing in one vertex only,
either in side $i$ or side $r+1$. 

Again, we use the truncated projective plane \tpp{r} \ for our construction in Theorem~\ref{MainTheorem}. 
 Let us fix an arbitrary line $S\in \tpp{r}$, and
an appropriate selection of lines $F_i\neq S$ with $s_i \in F_i$. 

\medskip

\begin{lemma}\label{LemmaMaximal}
For $r\geq 7$, let $\mathcal G$ be an intersecting $(r{+}1)$-partite hypergraph containing 
$\mathcal H =\mathcal H(\tpp{r}, S, F_1, \dots, F_r)$ 
and $E\in \mathcal G \setminus \mathcal H$.
Then $|E|=r+1$ and there is some $i\leq r$ and vertex $v$ such that either 
$E=F_i+v$ (of type $1$) or $E=F_i-s_i+v_i +v$ (of type $2$).
\end{lemma}

\begin{proof} 
Suppose the statement is false. Let $\mathcal G$ and $E\in \mathcal G$ provide a
counterexample such that $|E\cap V(\mathcal H)|$ is as large as possible amongst 
all counterexamples.

\begin{claim}
$|E\cap V(\mathcal H)|=r+1$.
\end{claim}
\begin{proof}

We know that $\tau(\mathcal H)=r$ and $\mathcal G \supseteq \mathcal H$ is intersecting.
Therefore, $E\cap V(\mathcal H)$ must be a cover of $\mathcal H$.

First suppose that $|E| < r+1$. Then $E$ must be of size $r$ and fully contained in $V(\mathcal H)$. Let $V_j$ be the side in which $E$ has no vertex. We will show that we can add a vertex of $V_j$ to $E$ such that we do not create an edge of type 1 or 2.
Let $i$ be an index such that $|E\cap F_i|$ is as large as possible. Let $x$ be a vertex in $V_j\cap V(\mathcal H)$ which is not in $F_i+v_i$. We will show that $E+x$ is not of type 1 or 2.
If $|E\cap F_i|\leq r-3$, then $|(E+x)\cap F_k|\leq r-2$ for all $k$, which implies $E+x$ is not of type 1 or 2. If $|E\cap F_i|\geq r-2$, then since $r\geq 5$, we have that $|E\cap F_k|\leq r-3$ for all $k\neq i$. As before, this implies that $|(E+x)\cap F_k|\leq r-2$ for all $k\neq i$, and so the only way $E+x$ could be of type 1 or 2 is if $E+x=F_i+x$ or $E+x=F_i-s_i+v_i+x$. However, in either case we would obtain that $E\in \mathcal H$ contradicting the assumption of the lemma.
Therefore, replacing $E$ by $E+x$  increases the size of the intersection $|E\cap V(\mathcal H)|$, contradicting its 
maximality.

{}From now on we assume that $|E| =r+1$.

Suppose now that there is an  index $i$ such that $E\cap V_i \cap V(\mathcal H) =\emptyset$ and let $x\in E\cap V_i$.
Since $r\geq 5$, there must be at least $4$ vertices $x_1, x_2, x_3, x_4\in V_i\cap V(\mathcal H)$. 
Let $E_i=E -x+x_i$ for $i=1, \dots, 4$ to get four sets, 
each having one more vertex in $V(\mathcal H)$ than $E$ has. 
We show that one of these edges $E_i$ is neither in ${\cal H}$ nor has type 1 or 2.
For such an edge, ${\cal G}' = {\cal H} \cup \{ E_i\}$ is
an intersecting hypergraph and for $E_i\in {\cal G}'$ we have
$|E_i\cap V(\mathcal H)| > |E\cap V(\mathcal H)|$ contradicting  
the maximality condition of the definition of $E$.

We check the required property of the $E_i$-s.
Notice that the maximum intersection size between a pair of ($r+1$)-edges of $\mathcal H$ is 2.
Therefore, at most one of the edges $E_1, \dots, E_4$ is in $\mathcal H$. 
Further, if for distinct $j$ and $k$ we had $E_j=F_a+y_a$ and $E_k=F_b+y_b$ 
for some $a,b\in \{ 1, \ldots , r\}$ and vertices $y_a$ and $y_b$, 
then we would have $E_j\cap E_k=E- x\subseteq 
(F_a+y_a)\cap (F_b+y_b)$. Since $|(F_a+y_a)\cap (F_b+y_b)|\leq 3$ for distinct $a$ and $b$, 
$|E_j\cap E_k|\geq r{-}1$ implies  that $F_a=F_b$ and hence $E=F_a+x$, contradicting our assumption that $E$ is not of type 1.
Similarly, if for distinct $j$ and $k$ we had $E_j=F_a-s_a+v_a+y_a$ and 
$E_k=F_b-s_b+v_b+y_b$, then we would have 
$E_j\cap E_k=E\setminus V_i=(F_a-s_a+v_a)\cap (F_b-s_b+v_b)$.
Since $|(F_a-s_a+v_a)\cap (F_b-s_b+v_b)|\leq 3$ for distinct $a$ and $b$, 
$|E_j\cap E_k|\geq r{-}1$ implies  that $F_a=F_b$ and hence $E=F_a-s_a+v_a+x$, 
contradicting our assumption that $E$ is not of type 2.
In summary, at most three of the edges $E_1, \dots, E_4$ do not
satisfy the required property and hence one of them does provide 
the contradiction in the end of the previous paragraph.
\end{proof}

Let $v_i$ be the vertex in $E\cap V_{r+1}$ and $B=E-v_i+s_i$. The set $B$ intersects 
every edge of $\tpp{r}$ as well as each of the sides $V_1, \dots, V_r$, i.e. 
$B$ is a blocking set  of ${\cal P}_r$.
Since $B$ has size at most $q+2=r+1$, it must contain a full line 
$L$ of ${\cal P}_r$. 

Let $p_i$ be the vertex in $E\cap V_i$. 
Notice that $B$ has only one vertex in each side, with the exception of $V_i$, where $B$ contains both 
$s_i$ and $p_i$ (although it is possible that $s_i=p_i$). 
Since $B$ contains the line $L$ we have either $B=L+s_i$ or $B=L+p_i$.
From the definition of $B$, we obtain that $E=L+v_i$ or $E=L-s_i+p_i+v_i$ holds.

Suppose that $E=L+v_i$. If $L=F_j$ for some $j$, then $E$ is of type 1, and we are done. 
Otherwise let $s_j$ be $L\cap S$. 
If $j=i$, then $E=\hat L\in  \mathcal H(\tpp{r}, S, F_1, \dots, F_r)$, contradicting our assumption that 
$E$ is not from $\mathcal H(\tpp{r}, S, F_1, \dots, F_r)$. 
If $j\neq i$, then $L$ is disjoint from the edge $F_j-s_j+v_j\in \mathcal H$, contradicting 
$\mathcal G$ being intersecting.

Suppose that $E=L-s_i+v_i+p_i$. Since we are not in the previous case, we can assume that 
$s_i\neq p_i$. If $L=F_i$, then $E$ is of type 2, and we are done. 
Otherwise $E\cap F_i=\emptyset$ since $L\cap F_i=s_i$ and $v_i\notin F_i$, contradicting  $\mathcal{G}$ being intersecting.
\end{proof}

\subsection{The Ryser poset}
 
Extremal $r$-uniform hypergraphs for Ryser's Conjecture possess two properties for some integer~$\nu$: 
they have matching number at most $\nu$  and vertex cover number at least $(r{-}1)\nu$. 
The first of these is a monotone decreasing property, while the second is monotone increasing. 
This suggests the definition of a poset structure on the family of extremal $r$-graphs. 
While a similar poset can be defined 
for arbitrary matching number $\nu$, here we restrict ourselves for the 
intersecting case, as that is already complicated enough. 
The {\em $r$-Ryser poset} ${\cal R}_r$ contains all intersecting $r$-Ryser hypergraphs (up to isomorphism), that is, all $r$-partite, $r$-uniform
hypergraphs which are intersecting and have vertex cover number $r{-}1$. 
The poset relation $<$ is given by the sub-hypergraph relation. 

The $2$-Ryser poset ${\cal R}_2$ 
is an infinite chain of stars with $K_2$ as its minimal element.
The poset ${\cal R}_3$ was determined in \cite{HNT} (for arbitrary matching numbers). In the intersecting case, ${\cal R}_3$   has a unique minimal element:
the $3$-graph $R$ obtained from $\tpp{3}$ by deleting one of its edges. Above $R$
in the poset there is of course $\tpp{3}$, but as well the direct product of three 
infinite chains  with $R$ as their common minimal element.  This is because
one could add to $R$ an arbitrary number 
of  twin copies of each of its three degree-one vertices, without
losing its extremality for the $3$-uniform Ryser Conjecture.

To characterize $r$-Ryser hypergraphs, it is necessary
to have a full understanding of the minimal and maximal elements of the Ryser poset. 
What a minimal element in some Ryser poset should be is quite clear: an $r$-Ryser hypergraph with the
property that the deletion of any hyperedge {\em reduces} the vertex cover number. These were 
discussed in Section~\ref{sec:manyextremals}.
Naturally, one would want to define the maximal elements of the Ryser poset 
as the hypergraphs with the property that the addition of any new hyperedge {\em increases} the matching 
number. With this definition however we would not have the very much desirable property 
that an intersecting $r$-Ryser hypergraph
is always contained in a maximal one. Indeed, as we saw above it is sometimes possible to add infinitely many ``twin copies'' of an edge to an $r$-Ryser hypergraph without ever reaching a maximal element. 

This issue can be resolved by allowing countably infinite $r$-Ryser hypergraphs into the poset.   
When considering infinite hypergraphs however, a technical issue arises, since it is possible 
to have two non-isomorphic hypergraphs which are {\em both} subgraphs of each other. 
To circumvent this, we consider an equivalence relation $\sim$ on all $r$-uniform countable hypergraphs where 
$\mathcal H\sim \mathcal H'$ if $\mathcal H$ is contained in $\mathcal H'$ and also $\mathcal H'$ is contained in $\mathcal H$. 
Then equivalence classes of intersecting extremal hypergraphs form a poset under containment. 
Note that using the Sunflower Lemma it is not difficult to see that the Ryser poset has finitely many maximal elements (an argument for why this is true can be found in Chapter 2 of \cite{AK}).
Combining Lemmas~\ref{CountingS} and~\ref{LemmaMaximal} we obtain the following.

 \begin{figure}
 \centering\vspace{-1.7cm}
  \noindent\hspace{-0.7cm}
  \includegraphics[width=0.78\paperwidth]{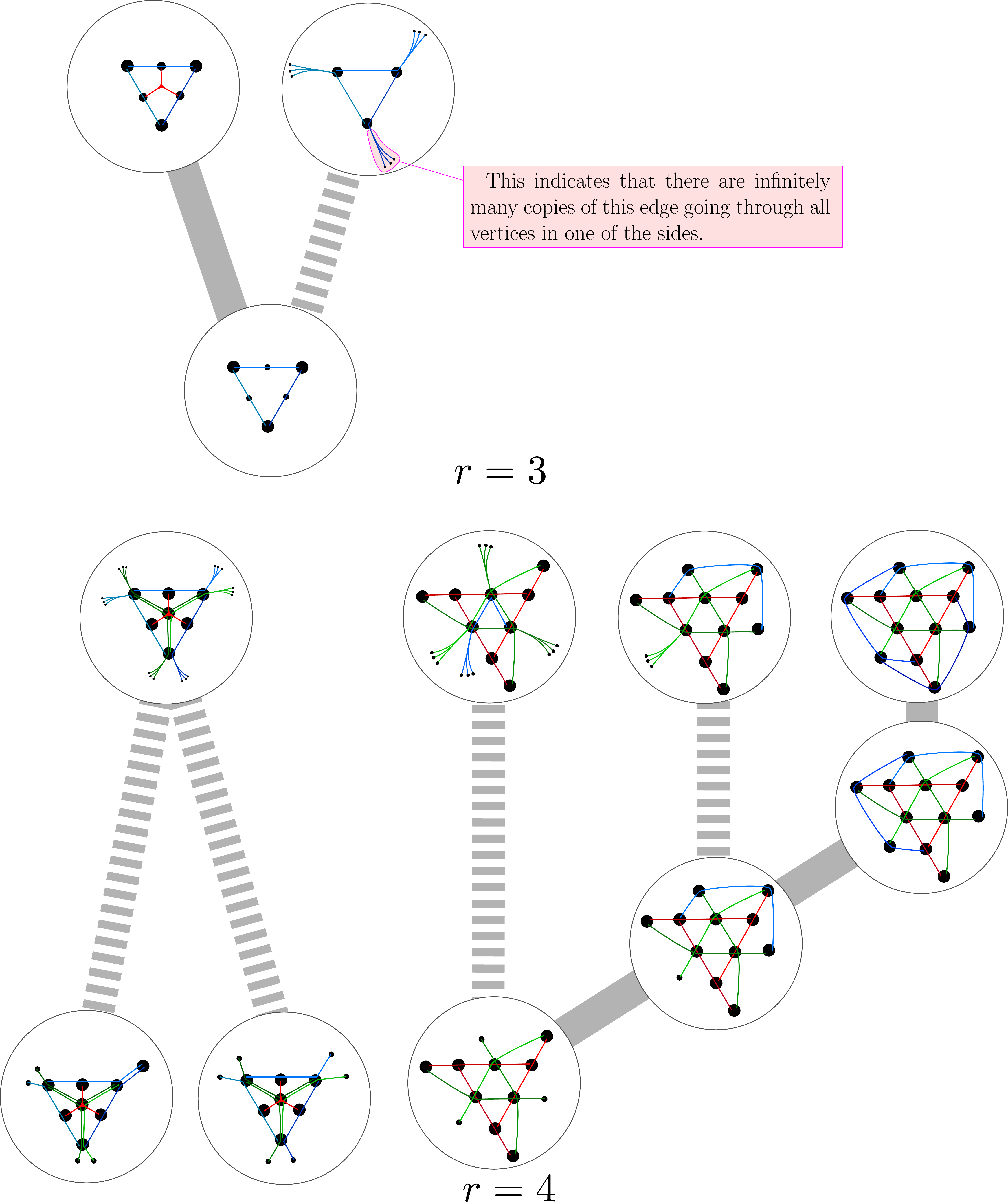}
\caption{The Ryser posets for $r=3$ and $4$. 
Each large circle contains a hypergraph in the Ryser poset: the small circles represent vertices of the hypergraphs, and the lines represent edges of the hypergraphs (with an edge containing a vertex exactly when the corresponding  line passes through a small circle). For clarity, we do not draw the hypergraphs in an obviously $r$-partite way. It is routine to check that the hypergraphs are actually $r$-partite. \\
In each case $r=3$ and $4$, we draw the Hasse diagram for the Ryser poset. The thick solid grey lines between the hypergraphs represent comparability in the Hasse diagram. Recall that the Ryser poset is infinite. The thick dashed lines represent infinite chains in the poset ending in an infinite hypergraph which is a maximal element in the poset.\\
The top and bottom rows in each diagram represent the maximal and
minimal elements of the respective poset.}
\end{figure}
\label{FigurePoset}

\medskip

\begin{corollary}\label{TheoremCountingMaximal}
There is an infinite sequence of integers $r$, for which there are 
$\exp (r^{0.5 -o(1)}) $ non-isomorphic maximal elements in the $r$-Ryser poset.
\end{corollary}
The maximal elements in Corollary~\ref{TheoremCountingMaximal} come from the hypergraphs 
 $\mathcal H(\tpp{r}, S, F_1, \dots, F_r)$, adding infinitely many vertices to each partition, and then  adding all the (infinitely many) copies of edges of the 
 forms $F_i+v$, $v\in V_{r+1}$, and $F_i-s_i+v_i+v$, $v\in V_i$, for every
 $i=1, \dots, r$. From Lemma~\ref{LemmaMaximal} we have that the resulting infinite hypergraphs are maximal since the only edges which can be added to them must be of the form $F_i+v$ and $F_i-s_i+v_i+v$, where $v$ is a new vertex.

 \section{Concluding remarks} \label{sec:conclusion}

\paragraph{1. No truncated projective planes.} 
While it is only conjectured that projective planes, and hence truncated projective planes,
do not exist for orders other than prime powers, one can show that for infinitely many
of the uniformities $r=p+2$ our constructions deals with, a truncated 
projective plane cannot exist.
For this let us first note that there are infinitely many primes $p$ of the form $8k+5$ by 
Dirichlet's Theorem. For such primes, by a well-known consequence of 
Fermat's Theorem on sums of two squares, $p+1 = 2(4k+3)$ 
is {\em not} the sum of two squares since its prime factorization must have a prime factor
of the form $4m + 3$ with an odd exponent.
Moreover $p+1$ is congruent to 
$2$ modulo $4$, so by the Bruck-Ryser Theorem there exists no projective plane
of order $p+1$. Hence our constructions of uniformity $p+2$ provide extremal
hypergraphs, where truncated projective planes do not exist.

\paragraph{2. The Ryser poset.} 
Theorem~\ref{counting} and Corollary~\ref{TheoremCountingMaximal} show that the Ryser poset (containing all the intersecting $r$-Ryser hypergraphs) 
has many non-isomorphic 
maximal and minimal elements, giving further support for the difficulty of Ryser's Conjecture. 
A better understanding of the Ryser poset seems essential to approach Ryser's Conjecture in general. 
With Theorem~\ref{counting} and Corollary~\ref{TheoremCountingMaximal} we made the first steps in this direction.
There are several other natural extremal problems that arise.

\medskip

\begin{problem}
What is the smallest and largest number of edges in a  
minimal intersecting extremal $r$-Ryser hypergraph?
\end{problem}
For the number of edges in a maximal $r$-Ryser hypergraph,  
truncated projective planes and our construction provide examples with 
$\Theta (r^2)$ edges. 

It was observed in \cite{ABW} that if a projective plane of order $r-1$ exists then 
one can construct quite sparse intersecting Ryser hypergraphs randomly.
Kahn~\cite{KAH} proved that with high probability a 
randomly chosen $22 r \log r$ lines of $\pp{r}$ cannot be covered with less than $r$ points.
This construction immediately  implies  a $O(r\log r)$ upper bond on the minimal size of an intersecting 
extremal $r$-Ryser hypergraph. It is an open question of Mansour, Song, and
Yuster~\cite{MSY} whether there exists one with $O(r)$ edges.

It would also be interesting to decide whether there exists maximal intersecting extremal  $r$-Ryser hypergraphs with sub-quadratic number of edges.

The $2$-Ryser poset ${\cal R}_2$ and the $3$-Ryser poset ${\cal R}_3$ are well understood (in fact a complete description of extremal hypergraphs is known even when not restricting oneself to intersecting hypergraphs \cite{HNT}). 
For $r=4$ we were able to completely determine the $4$-Ryser
poset. This was done by first computationally finding all minimal
intersecting $4$-partite hypergraphs (Which was completed by the first author in ~\cite{AK}). There are 3 minimal
hypergraphs in the $4$-Ryser poset. After finding the minimal elements
it is easy to check which edges can be added to them in order to find
the full $4$-Ryser poset. See Figure~1 for a diagram of the cases
$r=3$ and $r=4$. 
Finding the Ryser poset for $r=5$ seems to be a hard problem.

\paragraph{3. Asymptotic Ryser.}
Since there is still no construction of $r$-partite $r$-uniform 
hypergraphs with $\tau(\mathcal{H})=r{-}1$ for all $r$, 
it would be interesting to investigate hypergraphs with cover number close to $r{-}1$. 
Notice that using trunctated projective planes it is possible to construct for \emph{every} $r$ an 
 $r$-partite intersecting hypergraph with  $\tau(\mathcal{H})=r-o(r)$
by adding $s = o(r)$ new vertices to each edge in $\tpp{r-s}$. 
(The necessary prime $r-s-1$ exists by the known estimates on gaps between 
consecutive primes).

Any family of graphs satisfying $\tau(\mathcal{H})=r-o(r)$, which is different from the projective plane construction would already be interesting. 
We set the following problem to motivate further research.

\medskip

\begin{problem}
For some fixed constant $c$ and every $r$ construct an $r$-uniform $r$-partite intersecting hypergraph with $\tau(\mathcal{H})=r-c$.
\end{problem}

\emph{Note added in proof}: This has been recently solved by Haxell and Scott \cite{HS2}. For sufficiently large $r$ they constructed $r$-uniform hypergraphs with  $\tau(\mathcal{H})=r-4$.

\paragraph{4. Non-intersecting Ryser.}
We constructed intersecting extremal hypergraphs for Ryser's Conjecture.
It is easy to construct extremal hypergraphs with matching number equal to $\nu$ simply by taking $\nu$ vertex-disjoint copies of an intersecting extremal hypergraph. 
A natural question is whether all extremal hypergraphs for Ryser's Conjecture can be built in a similar fashion out of intersecting ones. 
For $r=3$ we know that the answer is ``yes'' --- Haxell, Narins, and Szab\'o~\cite{HNT} showed that for any $3$-partite 
hypergraph $\hyper H$ with $\tau(\hyper H)=2\nu(\hyper H)$ contains $\nu(\hyper H)$ 
vertex-disjoint intersecting hypergraphs with cover number $2$.

Recently it was shown by Abu-Khazneh and Pokrovskiy that there exist $4$-partite  Ryser-hypergraphs with matching number $2$, that don't contain vertex-disjoint copies of the edge-minimal $4$-partite intersecting Ryser-hypergraphs. However, they still conjecture a generalisation of the tripartite Ryser-hypergraphs classification using a notion of vertex-minimality, which they back by computational evidence. Chapters 5 and 6 of \cite{AK} contain more details on these insights.

\medskip

\begin{problem}
What is the correct way -- if at all possible -- of generalising the classification of tripartite Ryser-hypergraphs?
\end{problem}

\section*{Acknowledgment}
We are indebted to Penny Haxell and Bal\'azs Patk\'os for fruitful
discussions and support.
We would like to thank Tam\'as Sz\H onyi for pointing out reference \cite{BSzT}

\bibliographystyle{plain}
\bibliography{ryser_bib_v4}
\newpage

\end{document}